\documentclass[12pt]{amsart}
\usepackage [latin1]{inputenc}
\usepackage{amssymb}
\usepackage[all]{xy}

\usepackage{pgf, tikz}
\definecolor {processblue}{cmyk}{0.96,0,0,0}
\input xy
\xyoption{all}
\usepackage{mathtools}
\usepackage{amsfonts}
\usepackage{amsthm}

\usepackage{amscd}
\newtheorem{lemma}{Lemma}[section]
\newtheorem{corollary}[lemma]{Corollary}
\newtheorem{theorem}[lemma]{Theorem}
\newtheorem{proposition}[lemma]{Proposition}

\theoremstyle{definition}
\newtheorem{remark}[lemma]{Remark}
\newtheorem{definition}[lemma]{Definition}

\newtheorem{example}[lemma]{Example}

\newcommand{\bsm}{\begin{smallmatrix}}
\newcommand{\esm}{\end{smallmatrix}}
\newcommand{\bbm}{\begin{matrix}}
\newcommand{\ebm}{\end{matrix}}

\begin{document}

\title{right $n$-Nakayama algebras and their representations}

\author{Alireza Nasr-Isfahani}
\address{Department of Mathematics\\
University of Isfahan\\
P.O. Box: 81746-73441, Isfahan, Iran\\ and School of Mathematics, Institute for Research in Fundamental Sciences (IPM), P.O. Box: 19395-5746, Tehran, Iran}
\email{nasr$_{-}$a@sci.ui.ac.ir / nasr@ipm.ir}
\author{Mohsen Shekari}
\address{Department of Mathematics\\
University of Isfahan\\
P.O. Box: 81746-73441, Isfahan, Iran}
\email{mshekari@sci.ui.ac.ir}



\begin{abstract} In this paper we study right $n$-Nakayama algebras. Right $n$-Nakayama algebras appear naturally in the study of representation-finite algebras. We show that an artin algebra $\Lambda$ is representation-finite if and only if $\Lambda$ is right $n$-Nakayama for some positive integer $n$. We classify hereditary right $n$-Nakayama algebras. We also define right $n$-coNakayama algebras and show that an artin algebra $\Lambda$ is right $n$-coNakayama if and only if $\Lambda$ is left $n$-Nakayama. We then study right $2$-Nakayama algebras. We show how to compute all the indecomposable modules and almost split sequences over a right $2$-Nakayama algebra. We end by classifying finite dimensional right $2$-Nakayama algebras in terms of their quivers with relations.

\end{abstract}

\maketitle

\section{Introduction}
Let $R$ be a commutative artinian ring. An $R$-algebra $\Lambda$ is called an artin algebra if
$\Lambda$ is finitely generated as an $R$-module.
Given an artin algebra $\Lambda$, it is a quite natural question
to ask for the classification of all the indecomposable finitely generated right $\Lambda$-modules. Only for few
classes of algebras such a classification is known, one of the first such class were the Nakayama algebras. A Nakayama algebra $\Lambda$ is an algebra such that the
indecomposable projective right $\Lambda$-modules as well as the indecomposable injective right $\Lambda$-modules are uniserial. This then implies
that all the indecomposable right $\Lambda$-modules are uniserial. Nakayama algebras were studied by Tadasi Nakayama who called them generalized uniserial rings \cite{Na, Na1}. A right $\Lambda$-module $M$ is called uniserial if it has a unique composition series. Uniserial modules are the simplest indecomposables and this makes it interesting to understand their role in the category $mod(\Lambda)$ of finitely generated right $\Lambda$-modules.

An artin algebra $\Lambda$ is said to be representation-finite, provided there
are only finitely many isomorphism classes of indecomposable right $\Lambda$-modules.
In representation theory, representation-finite algebras are of
particular importance since in this case one has a complete combinatorial description
of the module category in terms of the Auslander-Reiten quiver. The class of Nakayama algebras is one of the fundamental
classes of representation-finite algebras whose representation theory is completely
understood.

In this paper we introduce the notion of $n$-factor serial modules. We say that a non-uniserial right $\Lambda$-module $M$ of length $l$ is $n$-factor serial ($l\geq n>1$), if $\frac{M}{\mathit{rad}^{l-n}(M)}$ is uniserial and $\frac{M}{\mathit{rad}^{l-n+1}(M)}$ is not uniserial. In some sense, $n$ is an invariant that measures how far $M$ is from being uniserial. We say that an artin algebra $\Lambda$ is right $n$-Nakayama if every finitely generated indecomposable right $\Lambda$-module is $i$-factor serial for some $1 \leqslant i \leqslant n$ and there exists at least one indecomposable $n$-factor serial right $\Lambda$-module. We show that right $n$-Nakayama algebras create a nice partition for representation-finite algebras. More precisely, we show that an artin algebra $\Lambda$ is representation-finite if and only if $\Lambda$ is right $n$-Nakayama for some positive integer $n$. By using this fact we give another proof of the first Brauer-Thrall conjecture for artin algebras. Recall that the first Brauer-Thrall conjecture as established by Roiter \cite{Ro} asserts that, any artin algebra with infinitely many
isomorphism classes of indecomposable modules of finite length has indecomposable modules of arbitrary large finite length. We show that any artin algebra with infinitely many
isomorphism classes of indecomposable modules of finite length has indecomposable $n$-factor serial modules for arbitrary large positive integer $n$. This improves the assertion of the first
Brauer-Thrall conjecture. The second Brauer-Thrall conjecture says that, an artin algebra $\Lambda$ of cardinality $\aleph\geq \aleph_0$, where $\aleph_0$ stands for the cardinality of a countable set,  is representation-infinite if and only if there exist infinitely many positive integers $n_i$ with $\aleph$ non-isomorphic indecomposable right $\Lambda$-modules of length $n_i$. The conjecture was solved affirmatively by Nazarova and Roiter \cite{NR} (see also \cite{B, F} and \cite{R}) for finite dimensional algebras over algebraically closed fields, but the conjecture is still open in general. We show that the second Brauer-Thrall conjecture is
equivalent to the statement that, an artin algebra $\Lambda$ of cardinality $\aleph\geq \aleph_0$ is either right $n$-Nakayama for some positive integer $n$ or there exist infinitely many positive integers $n_i$ with $\aleph$ non-isomorphic indecomposable $n_i$-factor serial right $\Lambda$-modules.

In the partition of representation-finite algebras, which is created by right $n$-Nakayama algebras, after the class of Nakayama algebras the second part is the class of right $2$-Nakayama algebras. Let $\Lambda$ be an artin algebra which is not a Nakayama algebra. We show that $\Lambda$ is right $2$-Nakayama if and only if every indecomposable non-projective right $\Lambda$-module is uniserial. By using this fact we classify indecomposable modules and almost split sequences over right $2$-Nakayama algebras. In \cite{NS} and \cite{NS1} we study right $3$-Nakayama and right $4$-Nakayama algebras. We compute indecomposable modules and almost split sequences over right $3$-Nakayama and right $4$-Nakayama algebras. Also we classify finite dimensional right $3$-Nakayama and right $4$-Nakayama algebras in terms of their quivers with relations.

The paper is organized as follows. In Section 2 we first introduce the $n$-factor serial right modules and right $n$-Nakayama algebras. Then by using the partition of representation-finite algebras, which creates by right $n$-Nakayama algebras, we improve the assertion of the first
Brauer-Thrall conjecture. We also provide an equivalence statement for the second Brauer-Thrall conjecture.

In Section 3 we give a characterization of right $n$-Nakayama hereditary finite dimensional algebras. Let $\Lambda=kQ$ be a finite dimensional hereditary algebra. If $Q$ is a quiver of type $\mathbb{A}_n$, then we show that $\Lambda$ is either Nakayama or right $(n-1)$-Nakayama or right $n$-Nakayama, depends to the orientation of $Q$. If $Q$ is a quiver of type $\mathbb{D}_n$, $\mathbb{E}_6$, $\mathbb{E}_7$ or $\mathbb{E}_8$, then we show that $\Lambda$ is $(2n-3)$-Nakayama, $11$-Nakayama, $17$-Nakayama or $29$-Nakayama, respectively.

In Section 4 we introduce the $n$-cofactor serial right modules and right $n$-coNakayama algebras. By using the standard duality we show that $M$ is an $n$-factor serial left $\Lambda$-module if and only if $\mathit{D}(M)$ is an $n$-cofactor serial right $\Lambda$-module. Then we show that an artin algebra $\Lambda$ is left $n$-Nakayama if and only if $\Lambda$ is right $n$-coNakayama.

 In the final section, we focus on right $2$-Nakayama algebras. We show that if every indecomposable non-projective right $\Lambda$-module is uniserial, then $\Lambda$ is either Nakayama or right $2$-Nakayama algebra. Also we give a characterization of indecomposable modules and almost split sequences over a right $2$-Nakayama algebra. Finally we describe finite dimensional right $2$-Nakayama algebras in terms of their quivers with relations.

\subsection{Notation }

Let $\Lambda$ be an artin algebra, we denote by $mod (\Lambda)$ the category of finitely generated right $\Lambda$-modules, $ind(\Lambda)$ the set of indecomposable right $\Lambda$-modules, $ind.proj(\Lambda)$ the set of indecomposable projective right $\Lambda$-modules and $ind.inj(\Lambda)$ the set of indecomposable injective right $\Lambda$-modules, one from each isomorphism class. Throughout this paper all modules are finitely generated right $\Lambda$-modules and all fields are algebraically closed fields unless otherwise stated. For a $\Lambda$-module $M$, we denote by $soc(M)$, $top(M)$, $rad(M)$, $\textit{l}(M)$, $\textit{ll}(M)$ and $\mathbf{dim} M$ its socle, top, radical, length, Loewy length and dimension vector, respectively.
Let $Q=(Q_0, Q_1, s, t)$ be a quiver and $\alpha:i\rightarrow j$ be an arrow in $Q$. One introduces a formal inverse $\alpha^{-1}$ with $s({\alpha}^{-1}) = j$
and $t(\alpha^{-1}) = i$. An edge in $Q$ is an arrow or the inverse of an arrow. To each vertex $i$ in $Q$,
one associates a trivial path, also called trivial walk, $\varepsilon_i$ with $s(\varepsilon_i) = t(\varepsilon_i) = i$. A non-trivial
walk $w$ in $Q$ is a sequence of edges $w=c_1c_2\cdots c_n$ such that $t(c_i) = s(c_{i+1})$ for each $i$, whose inverse $w^{-1}$ is defined to be
the sequence $w^{-1}=c_n^{-1}c_{n-1}^{-1}\cdots c_1^{-1}$. A walk $w$ is called reduced if $c_{i+1}\neq c_i^{-1}$ for each $i$. For $i\in Q_0$, we denote by $i^+$ and $i^-$ the set of arrows starting in $i$ and the set of arrows ending in $i$, respectively.

\section{right $n$-Nakayama algebras}

\begin{definition}
Let $\Lambda$ be an artin algebra and $M$ be a right $\Lambda$-module of length $l$.
\begin{itemize}
\item[$(1)$] $M$ is called  $1$-factor serial (uniserial) if $M$ has a unique composition series.
\item[$(2)$] Let $l\geq n>1$. We say that $M$ is $n$-factor serial if $\frac{M}{\mathit{rad}^{l-n}(M)}$ is uniserial and $\frac{M}{\mathit{rad}^{l-n+1}(M)}$ is not uniserial.
\end{itemize}
\end{definition}

\begin{definition}
We say that an artin algebra $\Lambda$ is right $n$-Nakayama if every indecomposable right $\Lambda$-module is $i$-factor serial for some $1 \leqslant i \leqslant n$ and there exists at least one  indecomposable $n$-factor serial right $\Lambda$-module. We can define left $n$-Nakayama algebras analogously.
\end{definition}

\begin{remark}
\begin{itemize}\item[(1)] $\Lambda$ is a right $1$-Nakayama algebra if and only if $\Lambda$ is a left $1$-Nakayama algebra if and only if $\Lambda$ is a Nakayama algebra.
\item[(2)] Let $\Lambda$ be a right $n$-Nakayama algebra. It is not necessary that for every $i$ with $1\leq i\leq n$, there exists an indecomposable $i$-factor serial right $\Lambda$-module. For example, let $\Lambda$ be the $K$-algebra given by the quiver
$$\hskip .5cm \xymatrix@-4mm{
&{1}\ar @{<-}[d]\\
{4}\ar [r]&{3}\ar [r]& {2}}\hskip .5cm$$
Its Auslander-Reiten quiver is given by
$$
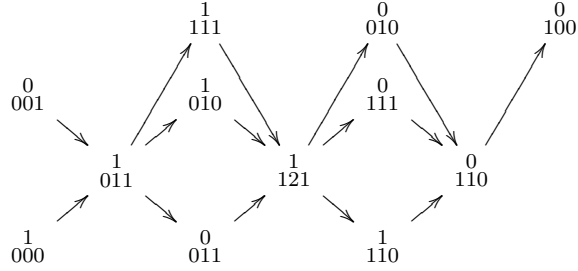
\xymatrix@-5mm{
&& {\bsm1\\111\esm}\ar[ddr]&&{\bsm0\\010\esm}\ar[ddr]&&{\bsm0\\100\esm} \\
{\bsm0\\001\esm}\ar[dr]&&{\bsm1\\010\esm}\ar[dr]&&{\bsm0\\111\esm}\ar[dr]\\
&{\bsm1\\011\esm}\ar[dr]\ar[uur]\ar[ur]&&{\bsm1\\121\esm}\ar[dr]\ar[uur]\ar[ur]&&{\bsm0\\110\esm}\ar[uur]\\
{\bsm1\\000\esm}\ar[ur]&&{\bsm0\\011\esm}\ar[ur]&&{\bsm1\\110\esm}\ar[ur]\\
}
$$
Where indecomposable modules are represented by their dimension vectors.
It is easy to see that indecomposable right $\Lambda$-modules $\bsm0\\010\esm, \bsm0\\100\esm, \bsm0\\001\esm, \bsm1\\010\esm, \bsm0\\111\esm, \bsm0\\110\esm, \bsm1\\000\esm, \bsm0\\011\esm$ and $\bsm1\\110\esm$ are uniserial, indecomposable right $\Lambda$-modules $\bsm1\\111\esm$ and $\bsm1\\011\esm$ are $2$-factor serial and an indecomposable right $\Lambda$-module $\bsm1\\121\esm$ is $5$-factor serial. Then $\Lambda$ is right $5$-Nakayama, but there is no indecomposable $3$-factor serial right $\Lambda$-module.
\end{itemize}
 \end{remark}

The following example shows that there is a right $n$-Nakayama algebra which is not left $n$-Nakayama.

 \begin{example}
Let $\Lambda$ be the $K$-algebra given by the quiver
$$\hskip .5cm \xymatrix@-4mm{
 2\ar@{<-}[r]&3 \ar[r]&1
 }\hskip .5cm $$
Its Auslander-Reiten quiver is given by
$$
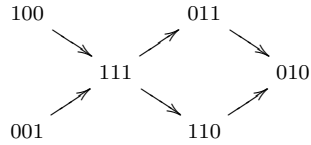
\xymatrix@-5mm{
{\bsm100\esm}\ar[dr]&& {\bsm011\esm} \ar[dr] \\
&{\bsm111\esm} \ar[dr]\ar[ur]&& {\bsm010\esm}\\
{\bsm001\esm}\ar[ur]&& {\bsm110\esm} \ar[ur] \\
}
$$
Indecomposable right $\Lambda$-modules $\bsm100\esm, \bsm011\esm, \bsm010\esm, \bsm001\esm$ and $\bsm110\esm$ are uniserial and an indecomposable right $\Lambda$-module $\bsm111\esm$ is $2$-factor serial. Then $\Lambda$ is right $2$-Nakayama. $\Lambda^{op}$ is the $K$-algebra given by the quiver
$$\hskip .5cm \xymatrix@-4mm{
 2\ar[r]&3 &1\ar[l]
 }\hskip .5cm $$
Its Auslander-Reiten quiver is given by
$$
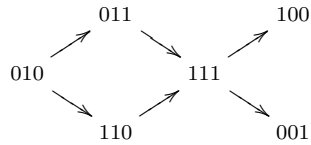
\xymatrix@-5mm{
&{\bsm011\esm}\ar[dr]&&{\bsm100\esm} \\
{\bsm010\esm} \ar[dr]\ar[ur]&& {\bsm111\esm}\ar[ur]\ar[dr]&\\
&{\bsm110\esm}\ar[ur]&& {\bsm001\esm}  \\
}
$$
Indecomposable right $\Lambda^{op}$-modules $\bsm011\esm, \bsm100\esm, \bsm010\esm, \bsm110\esm$ and $\bsm001\esm$ are uniserial and an indecomposable right $\Lambda^{op}$-module $\bsm111\esm$ is $3$-factor serial. Then $\Lambda$ is a left $3$-Nakayama algebra.
\end{example}

\begin{remark}\label{1}
Let $n>1$ and $M$ be an $n$-factor serial right $\Lambda$-module of length $l$. For each $1\leq i \leq l-n$, $\frac{\frac{M}{\mathit{rad}^{l-n}(M)}}{\frac{\mathit{rad}^{i}(M)}{\mathit{rad}^{l-n}(M)}}\cong \frac{M}{\mathit{rad}^{i}(M)}$ and $\frac{M}{\mathit{rad}^{l-n}(M)}$ is uniserial. Then $\frac{M}{\mathit{rad}^{i}(M)}$ is uniserial for each $1\leq i \leq l-n$.
\end{remark}

\begin{theorem}\label{2}
Let $\Lambda$ be an artin algebra, $M$ be a non-uniserial right $\Lambda$-module of length $l$ and $2\leq n\leq l$, be a positive integer. Then the following conditions are equivalent:
\begin{itemize}
\item[$(a)$] $M$ is $n$-factor serial.
\item[$(b)$] $\mathit{rad^i(M)}$ is local for each $0\leq i\leq l-n-1$ and $\mathit{rad}^{l-n}(M)$ is not local.
\item[$(c)$] For any non-zero submodule $N$ of $M$, if $\mathit{l}(N)>n$, then $N$ is local and if $\mathit{l}(N)=n$, then $N$ is not local.
\end{itemize}
\end{theorem}

\begin{proof} $(a)\Longrightarrow (b).$ Assume that there exists a positive integer $i$, $0\leq i \leq l-n-1$, such that $\mathit{rad^i(M)}$ is not local, so $\frac{M}{\mathit{rad^{i+1}(M)}}$ is not uniserial, which is a contradiction to the Remark \ref{1}.

$(b)\Longrightarrow (a).$ It is obvious.

$(b)\Longrightarrow (c).$ Let $N$ be a non-zero submodule of $M$ that $\mathit{l}(N)\geq n$. Then by $(b)$, there exists $0\leq i \leq l-n$, such that $\mathit{rad}^{i}(M)=N$. If $\mathit{l}(N)> n$,  then $N$ is local and if $\mathit{l}(N)=n$, then $N$ is not local.

$(c)\Longrightarrow (b).$ If $l=n$, then by assumption $M$ is not local and $(b)$ follows. Now assume that $l>n$. Then by assumption $M$ is local and $\mathit{l}(\mathit{rad}(M))=l-1$. If $n=l-1$, then $\mathit{rad}(M)$ is not local and $(b)$ follows. If $l-1>n$, then $\mathit{rad}(M)$ is local and $\mathit{l}(\mathit{rad}^2(M))=l-2$. Continuing in this way, we can see that  $\mathit{rad}^i(M)$ is local for each $0\leq i \leq l-n-1$ and
$\mathit{rad}^{l-n}(M)$ is not local.
\end{proof}

\begin{remark}\label{3}
Let $\Lambda$ be an artin algebra, $n>1$ be a positive integer and $M$ be an $n$-factor serial right $\Lambda$-module of length $l$.
If $\mathit{l}(\mathit{soc}(M))>n$, then by Theorem \ref{2} there exists a positive integer $i$, $1\leq i\leq l-n-1$, such that $\mathit{soc}(M)= \mathit{rad}^{i}(M)$. Thus $\mathit{rad}^{i+1}(M)=0$ and $\frac{M}{\mathit{rad}^{i+1}(M)}=M$ is not uniserial which is a contradiction to the Remark \ref{1}. Therefore $\mathit{l}\mathit{(soc}(M)) \leq n$ and since $\mathit{rad}^{l-n-1}(M)$ is local, $\mathit{soc}(M)\subseteq \mathit{rad}^{l-n}(M)$.
\end{remark}

\begin{corollary}\label{4}
Let $\Lambda$ be an artin algebra and $M$ be a right $\Lambda$-module of length $l>1$. Then $M$ is $l$-factor serial if and only if $M$ is not local.
\end{corollary}

\begin{proof} It follows from Theorem \ref{2}.
\end{proof}

\begin{corollary}\label{15}
Let $\Lambda$ be an artin algebra and $M$ be a right $\Lambda$-module of length $l$. If $M$ is not indecomposable, then $M$ is $l$-factor serial.
\end{corollary}
\begin{proof}
It follows from Corollary \ref{4}.
\end{proof}

\begin{lemma}\label{23}
Let $\Lambda$ be an artin algebra and $M$ be an indecomposable right $\Lambda$-module such that $\mathit{rad}^{2}(M)=0$. Then the following statements hold:
\begin{itemize}
\item[$(a)$] $M$ is uniserial if and only if $\mathit{l}(M)=1$ or $2$.
\item[$(b)$] Assume that $l=\mathit{l}(M)> 2$. If $M$ is not local, then $M$ is $l$-factor serial and if $M$ is local, then $M$ is $(l-1)$-factor serial.
\end{itemize}
\end{lemma}

\begin{proof} $(a)$. It is clear.

$(b)$. Assume that $l> 2$. If $M$ is not local, then by Corollary \ref{4}, $M$ is $l$-factor serial. Now assume that $M$ is local. Since $\mathit{rad}^{2}(M)=0$,  $\frac{M}{\mathit{rad}^{l-(l-1)+1}(M)}=\frac{M}{\mathit{rad}^{2}(M)}=M$ is not uniserial and $\frac{M}{\mathit{rad}(M)}$ is uniserial. Thus $M$ is $(l-1)$-factor serial.
\end{proof}

\begin{lemma}\label{5}
Let $\Lambda$ be an artin algebra, $M$ and $N$ be two right $\Lambda$-modules that $M$ is $t$-factor serial and $N$ is $s$-factor serial for some positive integer $s>1$. If there exists a proper epimorphism $f: M\longrightarrow N$, then $t> s$.
\end{lemma}

\begin{proof}
Let $M_1=\frac{M}{\mathit{ker}(f)}$ and $\mathit{l}(M)=l$. Then $\mathit{l}(M_1)=l-i$ for some positive integer $i>0$. Suppose that $s\geqslant t$. By definition, $\frac{M_1}{\mathit{rad}^{l-i-s} (M_1)}$ is uniserial and $\frac{M_1}{\mathit{rad}^{l-(i+s)+1} (M_1)}$ is not uniserial. Since $M$ is $t$-factor serial and $M_1$ is not uniserial, by Theorem \ref{2}, $\mathit{ker}(f)\subsetneqq \mathit{rad}^{l-t}(M)$. On the other hand, since $t< s+i$,
$\mathit{rad}^{l-t}(M)\subsetneqq \mathit{rad}^{l-(s+i)}(M)$ and $\mathit{rad}^{l-t}(M)\subseteq \mathit{rad}^{l-(s+i)+1}(M)$. Therefore $\frac{M_1}{\mathit{rad}^{l-(s+i)+1}(M_1)}\cong \frac{\frac{M}{\mathit{ker}(f)}}{\mathit{rad}^{l-(s+i)+1}(\frac{M}{\mathit{ker}(f)})}\cong \frac{M}{\mathit{rad}^{l-(s+i)+1}(M)} $ is uniserial which gives a contradiction. Then $t>s$ and the result follows.
\end{proof}

Note that in Lemma \ref{5} if $s=1$, then obviously we have $t\geq s$.

\begin{remark}
Let $\Lambda$ be a representation-finite artin algebra with radical square zero, $l_1=max\{\mathit{l}(M)|  \text{$M$ is an indecomposable non-local right $\Lambda$-module}\}$, $l_2=\\max\{\mathit{l}(P)|\text{$P$ is an indecomposable projective right $\Lambda$-module}\}$ and $l=max\{l_1, l_2-1\}$. Then by Lemmas \ref{23} and \ref{5}, $\Lambda$ is right $l$-Nakayama.
\end{remark}

\begin{theorem}\label{6}
Let $\Lambda$ be an artin algebra and $M$ be a non-uniserial right $\Lambda$-module of length $l$. Then the following conditions are equivalent:
\begin{itemize}
\item[$(a)$] $M$ is $n$-factor serial.
\item[$(b)$] $\max \lbrace n_i \vert 1\leq i\leq d, \frac{M}{S_i}\,  is\,$  ${n_i}$-factor $ serial\rbrace =n-1$, where $\mathit{soc}(M)=S_1^{i_1}\oplus S_2^{i_2}\oplus\cdots\oplus S_d^{i_d}$.
\end{itemize}
\end{theorem}

\begin{proof} $(a)\Longrightarrow (b).$ Suppose that $S$ is a simple submodule of $M$. By Lemma \ref{5}, $\frac{M}{S}$ is $t$-factor serial for some $t\leq n-1$. Now we show that there exists a simple submodule $S$ of $M$ such that $\frac{M}{S}$ is $(n-1)$-factor serial. By Remark \ref{3}, $\mathit{soc}(M)\subseteq \mathit{rad}^{l-n}(M)$. If $\mathit{rad}^{l-n+1}(M)\neq 0$, then there exists a simple submodule $S$ of $M$ that $S$ is a submodule of $\mathit{rad}^{l-n+1}(M)$. It is clear that $\frac{\frac{M}{S}}{\mathit{rad}^{(l-1)-(n-1)}(\frac{M}{S})}\cong \frac{M}{\mathit{rad}^{l-n}(M)} $ is uniserial and $\frac{\frac{M}{S}}{\mathit{rad}^{(l-1)-(n-1)+1}(\frac{M}{S})}\cong \frac{M}{\mathit{rad}^{l-n+1}(M)}$ is not uniserial. Thus $\frac{M}{S}$ is $(n-1)$-factor serial and the result follows. Now assume that $\mathit{rad}^{l-n+1}(M)=0$. Then $\mathit{soc}(M)= \mathit{rad}^{l-n}(M)$ and $\mathit{l}(\mathit{soc}(M))=\mathit{l}(\mathit{rad}^{l-n}(M))=n$. If $n=2$, then $\frac{M}{S}$ is uniserial, $M$ is $2$-factor serial and the result follows. If $n>2$, then there exists a simple submodule $S$ of $M$ such that $\frac{M}{S}$ is not uniserial and $\mathit{rad}^{(l-1)-(n-1)+1}(\frac{M}{S})=0$. Then $\frac{\frac{M}{S}}{\mathit{rad}^{(l-1)-(n-1)+1}(\frac{M}{S})}\cong\frac{M}{S}$ is not uniserial and $\frac{\frac{M}{S}}{\mathit{rad}^{(l-1)-(n-1)}(\frac{M}{S})}\cong \frac{M}{\mathit{rad}^{l-n}(M)} $ is uniserial. Therefore $\frac{M}{S}$ is $(n-1)$-factor serial and the result follows.

$(b)\Longrightarrow (a).$ Let $S$ be a simple submodule of $M$ that $\frac{M}{S}$ is $(n-1)$-factor serial. By Lemma \ref{5}, $M$ is $t$-factor serial for some $t\geqslant n$. Assume that $t> n$. Then by the proof of the first part, there exists a simple submodule $S$ of $M$ such that $\frac{M}{S}$ is $(t-1)$-factor serial which is a contradiction. Then $M$ is $n$-factor serial and the result follows.
\end{proof}

\begin{remark}\label{R6}
Let $M$ be a right $\Lambda$-module and $n>1$ be a positive integer. By Theorem \ref{6} and Lemma \ref{5}, $M$ is $n$-factor serial if and only if $M$ satisfies the following conditions:
\begin{itemize}
\item[$\left( \mathrm{i}\right) $]  $M$ is not $(n-1)$-factor serial.
\item[$\left( \mathrm{ii}\right) $] For any nonzero submodule $N$ of $M$, there exists $1\leq t \leq n-1$ such that $\frac{M}{N}$ is $t$-factor serial.
\item[$\left( \mathrm{iii}\right) $] There exists a submodule $L$ of $M$ such that $\frac{M}{L}$ is $(n-1)$-factor serial.
\end{itemize}
\end{remark}

\begin{corollary}\label{7}
Let $\Lambda$ be an artin algebra and $M$ be a non-uniserial right $\Lambda$-module of length $l$. Then the following conditions are equivalent:
\begin{itemize}
\item[$(a)$] $M$ is an $n$-factor serial right $\Lambda$-module.
\item[$(b)$] There exists a sequence of proper epimorphisms and right $\Lambda$-modules
 \begin{center}
$ M_0\buildrel {f_0} \over \longrightarrow M_1 \buildrel {f_{1}} \over \longrightarrow \cdots \buildrel {f_{n-3}} \over \longrightarrow M_{n-2}\buildrel {f_{n-2}} \over \longrightarrow M_{n-1}$
 \end{center}
 that $M_0=M$, $M_i$ is $(n-i)$-factor serial for each $1\leq i\leq n-1$, $f_{n-2}...f_0\neq 0$ and this sequence is the longest sequence with these properties.
\item[$(c)$] There exists a sequence of proper epimorphisms and right $\Lambda$-modules
 \begin{center}
$ M_0\buildrel {f_0} \over \longrightarrow M_1 \buildrel {f_{1}} \over \longrightarrow \cdots \buildrel {f_{n-3}} \over \longrightarrow M_{n-2}$
 \end{center}
 that $M_0=M$, $M_{n-2}$ is $2$-factor serial, $f_{n-3}...f_0\neq 0$ and this sequence is the longest sequence with these properties.
 \end{itemize}
\end{corollary}

\begin{proof} $(a)\Longrightarrow (b).$  By using Theorem \ref{6} we can construct such sequence.

$(b)\Rightarrow (c).$ It is obvious.

$(c)\Rightarrow (a).$ Since $f_{n-3}$ is a proper epimorphism and $M_{n-2}$ is $2$-factor serial, by Lemma \ref{5}, $M_{n-3}$ is
$t$-factor serial for some $t\geq 3$. If $t> 3$, then by using Theorem \ref{6} we can construct a sequence of proper epimorphisms and right $\Lambda$-modules with length greater than $n-2$ that satisfy the conditions of $(c)$, which gives a contradiction. Thus $M_{n-3}$ is $3$-factor serial. Continuing in this way, we can see that $M=M_0$ is $n$-factor serial.
\end{proof}

\begin{lemma}\label{9}
Let $\Lambda$ be an artin algebra, $M$ be a non-uniserial $m$-factor serial right $\Lambda$-module of length $l$ and $N$ be a submodule of $M$. Then the following statements hold:
\begin{itemize}
\item[$(a)$] If $\mathit{l}(N)\geq m$, then $N$ is $m$-factor serial.
\item[$(b)$] If $\mathit{l}(N)< m$, then $N$ is $n$-factor serial for some $n<m$.
\end{itemize}
\end{lemma}

\begin{proof} $(a)$. Assume that $\mathit{l}(N)\geq m$. By Theorem \ref{2}, there exists a positive integer $i$, $0\leq i\leq l-m$, such that $N=\mathit{rad}^{i}(M)$. Thus $\mathit{l}(N)=\mathit{l}(\mathit{rad}^{i}(M))=l-i$. Then $\frac{N}{\mathit{rad}^{l-i-m}(N)}$ is uniserial and $\frac{N}{\mathit{rad}^{l-i-m+1}(N)}$ is not uniserial. Therefore $N$ is $m$-factor serial.

$(b)$. It follows from the definition.
\end{proof}

\begin{corollary}\label{16}
Let $\Lambda$ be an artin algebra, $M$ and $N$ be two right $\Lambda$-modules and $f:N\rightarrow M$ be a monomorphism. If $M$ is $m$-factor serial, then $N$ is $n$-factor serial for some $n\leq m$.
\end{corollary}

\begin{proof}
It follows from Lemma \ref{9}.
\end{proof}

The following theorem is the main theorem of this section which gives a partition for representation-finite artin algebras.

\begin{theorem}\label{10}
Let $\Lambda$ be an artin algebra. Then $\Lambda$ is representation-finite if and only if there exists a positive integer $n$ such that $\Lambda$ is right $n$-Nakayama.
\end{theorem}

\begin{proof} Assume that $\Lambda$ is a representation-finite artin algebra and let $\{M_1, M_2, \cdots, M_{t}\}$ be a complete set of isomorphism classes of indecomposable right $\Lambda$-modules. Put $n= max\{n_i| 1\leq i\leq t, \text{$M_i$ is $n_i$-factor serial}\}$. Then $\Lambda$ is right $n$-Nakayama.

Now, let $\Lambda$ be a right $n$-Nakayama algebra. Assume that $\Lambda$ is representation-infinite. Then by \cite[Theorem 3.1]{A2}, there is a sequence of proper epimorphisms between indecomposable right $\Lambda$-modules such
\begin{center}
$\cdots M_{i+1}\buildrel {f_i}\over \longrightarrow M_{i} \buildrel{f_{i-1}} \over\longrightarrow M_{i-1}\buildrel{f_{i-2}}\over\longrightarrow\cdots \buildrel{f_2}\over \longrightarrow M_2\buildrel{f_1}\over\longrightarrow M_{1}$
\end{center}
that for every positive integers $i$ and $n$, $f_{i}...f_{i+n}\neq 0$.
If there is a positive integer $j$ such that $M_j$ is not uniserial, then by Lemma $\ref{5}$ for every positive integer $n$, $M_{j+n}$ is $t$-factor serial for some $t>n$ which gives a contradiction. Now assume that $M_i$ is uniserial for each $i$. It is known that uniserial right $\Lambda$-modules are of bounded length. Then there exists a positive integer $j$ such that for every $i\geqslant j$, $f_i$ is an isomorphism which gives a contradiction. Therefore $\Lambda$ is representation-finite and the result follows.
\end{proof}

Theorem \ref{10} shows that an artin algebra $\Lambda$ is either representation-finite or there exists an infinite sequence $n_1<n_2<n_3<...$ of positive integers such that for each $i$, there exists an $n_i$-factor serial indecomposable right $\Lambda$-module.\\

A right $\Lambda$-module $M$ is called colocal if $soc(M)$ is simple.

\begin{proposition}
Let $\Lambda$ be an artin algebra, $m=\\\max\{m^{'}| P\in ind.proj(\Lambda), \text{$P$ is $m^{'}$-factor  serial}\}$,
$s=\\\max\{s^{'}| I\in ind.inj(\Lambda), \text{$I$ is $s^{'}$-factor serial}\}$,
$t=\\\max\{t^{'}| M\in ind(\Lambda), \text{$M$ is non-local $t^{'}$-factor serial}\}$ and
$k=\\\max\{k^{'}| M\in ind(\Lambda),  \text{$M$ is non-colocal $k^{'}$-factor serial}\}$.
Then the following conditions are equivalent:
\begin{itemize}
\item[$(a)$] $\Lambda$ is right $n$-Nakayama.
\item[$(b)$] $n=\max\{m, t\}$.
\item[$(c)$]$n=\max\{s, k \}$.
\end{itemize}
\end{proposition}

\begin{proof} $(a)\Longleftrightarrow (b).$  Let $M$ be an indecomposable non-projective local right $\Lambda$-module and $P\rightarrow M$ be a projective cover of $M$. Since $M$ is local, $P$ is indecomposable. By Lemma \ref{5}, if $P$ is $s$-factor serial and $M$ is $t$-factor serial, then $s\geq t$ and the result follows.

$(a)\Longleftrightarrow (c).$ Let $M$ be an indecomposable non-injective colocal right $\Lambda$-module and $M\rightarrow I$ be an injective envelope of $M$. Since $M$ is colocal, $I$ is indecomposable. By Corollary \ref{16}, if $I$ is $s$-factor serial and $M$ is $t$-factor serial, then $s\geq t$ and the result follows.
\end{proof}

An artin algebra $\Lambda$ is called of local type if every indecomposable right $\Lambda$-module is local. $\Lambda$ is called of colocal type if every indecomposable right $\Lambda$-module is colocal and $\Lambda$ is called of local-colocal type if every indecomposable right $\Lambda$-module is either local or colocal.

\begin{remark}
Let $\Lambda$ be an artin algebra. Then the following statements hold:
\begin{itemize}
\item[$(a)$] Let $\Lambda$ be a representation-finite self-injective algebra. Assume that $m=\\ \max\{m^{'}|P\in ind.proj(\Lambda), \text{$P$ is an $m^{'}$-factor serial right $\Lambda$-module}\}$ and $t=\\ \max\{t^{'}|M\in ind(\Lambda), \text{$M$ is non-local and non-colocal right $t^{'}$-factor serial}\}$. Then $\Lambda$ is right $n$-Nakayama if and only if $n=\max\{m, t\}$.
\item[$(b)$] Let $\Lambda$ be of local type. Then $\Lambda$ is right $n$-Nakayama if and only if
$n=\max\{n_{i}|P_{i}\in ind.proj(\Lambda), \text{$P_{i}$ is $n_{i}$-factor serial}\}$.
\item[$(c)$] Let $\Lambda$ be of colocal type. Then $\Lambda$ is right $n$-Nakayama if and only if
$n=\max\{n_{i}|I_{i}\in ind.inj(\Lambda), \text{$I_{i}$ is $n_{i}$-factor serial}\}$.
\item[$(d)$] Let $\Lambda$ be of local-colocal type. Assume that
$m=\\\max\{m^{'}|P\in ind.proj(\Lambda), \text{$P$ is $m^{'}$-factor serial}\}$ and
$s=\\\max\{s^{'}|I\in ind.inj(\Lambda), \text{$I$ is $s^{'}$-factor serial}\}$. Then $\Lambda$ is right $n$-Nakayama if and only if $n=\max\{m, s\}$.
\end{itemize}
\end{remark}

\begin{lemma}\label{l17}
Let $\Lambda$ be an artin algebra and $M$ be a right $\Lambda$-module of length $t$ and Loewy length $n$. If $M$ is $m$-factor serial, then $\max\left\lbrace  1, \,t-n\right\rbrace \leq m\leq t\leq m+n-1$.
\end{lemma}

\begin{proof}
First by induction on $m$ we show that $\max\left\lbrace  1, \,t-n\right\rbrace \leq m\leq t$. If $m=1$, then the result is obvious. Assume that the result holds for any $m$-factor serial right $\Lambda$-module and let $M$ be an $(m+1)$-factor serial right $\Lambda$-module of length $t'$ and Loewy length $n'$. By Theorem \ref{6}, there exists a simple submodule $N$ of $M$ such that $\frac{M}{N}$ is $m$-factor serial. $\textit{ll}(\frac{M}{N})=n''$ for some positive integer $n''\leq n'$. By the induction hypothesis $max\left\lbrace  1,\textit{l}(\frac{M}{N})-n'' \right\rbrace \leq m \leq \textit{l} \left( \frac{M}{N}\right)$ which implies that, $max\left\lbrace  1,\textit{l}(M)-n'' \right\rbrace \leq m+1 \leq \textit{l} \left({M}\right)$. Then we have $max\left\lbrace  1,\textit{l}(M)-n' \right\rbrace \leq m+1 \leq \textit{l} \left({M}\right)$ and the result follows. Now by induction on $t$ we show that $t\leq m+n-1$. If $t=1$, then the result is obvious. Assume that the result holds for any right $\Lambda$-module of length $t$ and let $M$ be an $m'$-factor serial right $\Lambda$-module of length $t+1$ and Loewy length $n'$. If $m'=1$, then $n'=t+1$ and the result follows. Now assume that $m'>1$. By Theorem \ref{6}, there exists a simple submodule $N$ of $M$ such that $\frac{M}{N}$ is $(m'-1)$-factor serial. $\textit{ll}(\frac{M}{N})=n''$ for some positive integer $n''\leq n'$. By the induction hypothesis we have $t\leq m'-1+n''-1$. Therefore $t+1\leq m'+n''-1\leq m'+n'-1$ and the result follows.
\end{proof}

\begin{remark}\label{11}
Let $\Lambda$ be an artin algebra of Loewy length $n$ and $M$ be an $m$-factor serial right $\Lambda$-module of length $t$ and Loewy length $n'$. Then by using Lemma \ref{l17} we have $\max\left\lbrace  1, \,t-n\right\rbrace \leq m\leq t\leq m+n-1$.
\end{remark}

An easy consequence of Theorem \ref{10} and Remark \ref{11} gives another proof of the first Brauer-Thrall conjecture for artin algebras.

\begin{corollary}
An artin algebra is either representation-finite or there exist indecomposable modules with arbitrary large length.
\end{corollary}

In the following theorem, by using the notion of $n$-factor serial modules, we give an equivalence statement for the second Brauer-Thrall conjecture.

\begin{theorem}
 Let $\Lambda$ be an artin algebra of cardinality $\aleph \geq \aleph_0$, where $\aleph_0$ stands for the cardinality of a countable set. Then the following conditions are equivalent:
 \begin{itemize}
 \item[($a$)] $\Lambda$ is either representation-finite or there exists an infinite sequence of positive integers $n_i\in \mathbb{N}$ such that, for each $i$, there exist $\aleph$ non-isomorphic indecomposable right $\Lambda$-modules of length $n_i$.
 \item[($b$)] $\Lambda$ is either right $n$-Nakayama for some positive integer $n$ or there exists an infinite sequence of positive integers $n_i\in \mathbb{N}$ such that, for each $i$, there exist $\aleph$ non-isomorphic indecomposable $n_i$-factor serial right $\Lambda$-modules.
     \end{itemize}
\end{theorem}

\begin{proof}$(a)\Longrightarrow (b).$ If $\Lambda$ is representation-finite, then by Theorem \ref{10}, $\Lambda$ is right $n$-Nakayama for some positive integer $n$. Now let $n_1<n_2< ...$ be an infinite sequence of positive integers that for each $i$, there exist $\aleph$ non-isomorphic indecomposable right $\Lambda$-modules of length $n_i$. By Remark \ref{11}, every indecomposable right $\Lambda$-module of length $n_i$ is
 $m$-factor serial for some positive integer $m$, $ \max\left\lbrace  1, \,n_{i}-t\right\rbrace \leq m\leq n_i$, where $t=ll(\Lambda)$. Then there exists a positive integer $m_i$, $\max\left\lbrace  1, \,n_{i}-t\right\rbrace \leq m_i\leq n_i$ such that there exist $\aleph$ non-isomorphic indecomposable $m_i$-factor serial right $\Lambda$-modules of length $n_i$. Therefore there exists an infinite sequence of positive integers $n'_i\in \mathbb{N}$ such that, for each $i$, there exist $\aleph$ non-isomorphic indecomposable $n'_i$-factor serial right $\Lambda$-modules.

$(b)\Longrightarrow (a).$ If $\Lambda$ is right $n$-Nakayama for some positive integer $n$, then by Theorem \ref{10}, $\Lambda$ is representation-finite. Now assume that there exists an infinite sequence $m_1<m_2< ...$ of positive integers such that for any $m_i$, there exist $\aleph$ non-isomorphic indecomposable $m_i$-factor serial right $\Lambda$-modules. By Remark \ref{11}, every indecomposable $m_i$-factor serial right $\Lambda$-module is of length $l$ for some positive integer $l$, $m_i\leq l \leq m_i+t-1$, where $ll(\Lambda)=t$. Then there exists a positive integer $n'_i$, $m_i\leq n'_i \leq m_i+t-1$ such that there exist $\aleph$ non-isomorphic indecomposable $m_i$-factor serial right $\Lambda$-modules of length $n'_i$. Therefore there exists an infinite sequence of positive integers $n_i\in \mathbb{N}$ such that, for each $i$, there exist $\aleph$ non-isomorphic indecomposable right $\Lambda$-modules of length $n_i$ and the result follows.
\end{proof}

\begin{definition}
We say that an artin algebra $\Lambda$ is $n$-Nakayama if it is both right $n$-Nakayama and left $n$-Nakayama.
\end{definition}

A finite dimensional $K$-algebra $\Lambda=\frac{KQ}{I}$ is called special biserial algebra provided $(Q,I)$ satisfying the following conditions:
\begin{itemize}
\item[$(1)$] For any vertex $a\in Q_0$, $|a^+|\leq 2$ and   $|a^-|\leq 2$.
\item[$(2)$] For any arrow $\alpha \in Q_1$, there is at most one arrow $\beta$ and at most one arrow $\gamma$ such that $\alpha\beta$ and $\gamma\alpha$ are not in $I$.
\end{itemize}

Let $\Lambda=\frac{KQ}{I}$ be an special biserial finite dimensional $K$-algebra. A walk $w=c_1c_2\cdots c_n$ is called string of length $n$ if $c_i\neq c_{i+1}^{-1}$ for each $i$ and no subwalk of $w$ or its inverse is in $I$. In addition for any $a\in Q_0$ we have two strings of length zero, denoted by $1_{(a,1)}$ and $1_{(a,-1)}$. We have $s(1_{(a,1)})=t(1_{(a,1)})=s(1_{(a,-1)})=t(1_{(a,-1)})=a$ and $1_{(a,1)}^{-1}=1_{(a,-1)}$. A string $w=c_1c_2\cdots c_n$ with $s(w)=t(w)$ such that each
power $w^m$ is a string, but $w$ itself is not a proper power of any strings is called band. We denote by $\mathcal{S}(\Lambda)$ and $\mathcal{B}(\Lambda)$ the set of all strings of $\Lambda$ and the set of all bands of $\Lambda$, respectively. Let $\rho$ be the equivalence relation on $\mathcal{S}(\Lambda)$ which identifies every string $w$ with its inverse $w^{-1}$ and $\sigma$ be the equivalence relation on $\mathcal{B}(\Lambda)$ which identifies every band $w=c_1c_2\cdots c_n$ with the cyclically permuted bands $w_{(i)}=c_ic_{i+1}\cdots c_nc_1\cdots c_{i-1}$ and their inverses $w_{(i)}^{-1}$, for each $i$.
Butler and Ringel in \cite{BR} for each string $w$ defined a unique string module $M(w)$ and for each band $v$ defined a family of band modules $M(v,m,\varphi)$ with $m\geq 1$ and $\varphi\in Aut(K^m)$. Let $\widetilde{\mathcal{S}}(\Lambda)$ be the complete set of representatives of strings relative to $\rho$ and $\widetilde{\mathcal{B}}(\Lambda)$  be the complete set of representatives of bands relative to $\sigma$. The special biserial algebra $\Lambda=\frac{KQ}{I}$ is called string algebra if the ideal $I$ can be generated by zero relations. Butler and Ringel in \cite{BR} proved that, the modules $M(w)$, $w\in \widetilde{\mathcal{S}}(\Lambda)$ and the modules $M(v,m,\varphi)$ with $v\in \widetilde{\mathcal{B}}(\Lambda)$, $m\geq 1$ and $\varphi\in Aut(K^m)$ provide the complete list of pairwise non-isomorphic indecomposable $\Lambda$-modules.
Indecomposable $\Lambda$-modules are either string modules or band modules. If $\Lambda$ is a representation-finite string algebra, then all indecomposable $\Lambda$-modules are string modules.\\

In the following example, we show that for every positive integer $n>2$ there exists an $n$-Nakayama algebra.

\begin{example}
Let $t>1$ be a positive integer and $\Lambda_t$ be a $K$-algebra given by the quiver\\

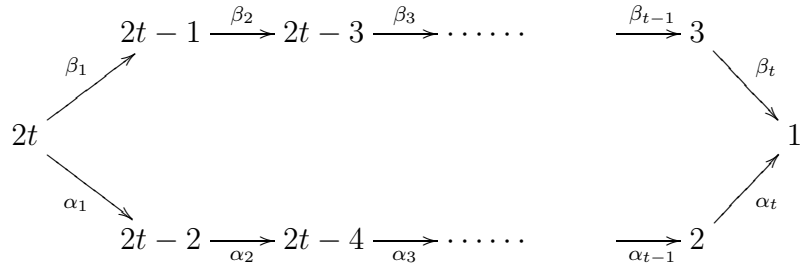
$$\hskip .5cm \xymatrix{
&{2t-1}\ar @{<-}[dl]_{\beta_{1}}\ar [r]^{\beta_{2}}&{2t-3}\ar[r]^{\beta_{3}}&\cdots\cdots&\ar[r]^{\beta_{t-1}}&3\ar[dr]^{\beta_{t}}\\
{2t}\ar[dr]_{\alpha_1}&&&&&&{1}\ar @{<-}[dl]^{\alpha_t}\\
&{2t-2}\ar [r]_{\alpha_{2}}&{2t-4}\ar[r]_{\alpha_3}&\cdots\cdots&\ar[r]_{\alpha_{t-1}}&2}\hskip .5cm$$
bounded by $\alpha_{1}\alpha_{2}=0$ and $\beta_{t-1}\beta_{t}=0$. $\Lambda_t$ is an string algebra which has no bands. Then $\Lambda_t$ is representation-finite and any indecomposable $\Lambda_t$-module is of the form $M(w)$, for some $w\in \widetilde{\mathcal{S}}(\Lambda_t)$. $P(2t)$ and $I(1)$ have the greatest dimension between indecomposable right $\Lambda_t$-modules. $\dim P(2t)=\dim I(1)=t+1$ and $I(1)$ is not local, then by Corollary \ref{4}, $I(1)$ is a $(t+1)$-factor serial right $\Lambda_t$-module. On the other hand $D(P(2t))$ and $D(I(1))$ have the greatest dimension between indecomposable left $\Lambda_t$-modules, where $D$ is the standard $K$-duality. $\dim D(P(2t))=\dim D(I(1))=t+1$ and $D(P(2t))$ is not local, then by Corollary \ref{4}, $D(P(2t))$ is a $(t+1)$-factor serial left $\Lambda_t$-module. Therefore $\Lambda_t$ is a $(t+1)$-Nakayama algebra.
\end{example}


\section{right $n$-Nakayama hereditary algebras}

Let $K$ be an algebraically closed field and $Q$ be a finite, connected and acyclic quiver.
Then $\Lambda=KQ$ is a finite dimensional hereditary $K$-algebra. By the Gabriel's Theorem \cite{G}, $\Lambda$ is representation-finite if and only if the underlying graph $\overline{Q}$ of $Q$ is one of the Dynkin diagrams $\mathbb{A}_n$, $\mathbb{D}_m$ with $m\geq 4$, $\mathbb{E}_6, \mathbb{E}_7$ and $\mathbb{E}_8$. If $\overline{Q}$ is a Dynkin diagram, then there is a bijection between the set of isomorphism classes of indecomposable right $\Lambda$-modules and the set of positive roots of the quadratic form $q_Q$ of $Q$.

Let $Q$ be a quiver such that the underlying graph $\overline{Q}$ of $Q$ is a Dynkin diagram $\mathbb{A}_n$ for some positive integer $n$. We have three cases:
\begin{itemize}
\item[$(1)$] For any $a\in Q_0$, $|a^{+}| \leq 1$ and $|a^{-}| \leq 1$. In this case $KQ$ is a Nakayama algebra.
\item[$(2)$] For any $a\in Q_0$,  $|a^{-}|\leq 1$ and there exists a vertex $b\in Q_0$ such that $|b^{+}|=2$.
\item[$(3)$] There exists a vertex $b\in Q_0$ such that $|b^{-}|=2$.
\end{itemize}

\begin{proposition}\label{14}
Let $\Lambda=KQ$ be a finite dimensional hereditary $K$-algebra such that the underlying graph $\overline{Q}$ of $Q$ is a Dynkin diagram $\mathbb{A}_n$, for some positive integer $n$. Then the following statements hold:
\begin{itemize}
\item[$(a)$] If for any $a\in Q_0$, $|a^-|\leq 1$ and there exists a vertex $b\in Q_0$ such that $|b^+|= 2$, then $\Lambda$ is right $(n-1)$-Nakayama.
\item[$(b)$] If there exists a vertex $b\in Q_0$ such that $|b^-|= 2$, then $\Lambda$ is right $n$-Nakayama.
\end{itemize}
\end{proposition}

\begin{proof} $(a).$  By the hypothesis, $Q$ is of the form
 $$\hskip .5cm \xymatrix{
{\bsm\bullet\esm}_2\ar@{<-}[r]^{\alpha_{2}}&{\bsm\bullet\esm}_4\cdots \ar@{<-}[r]^{\alpha_{n-3}}& {\bsm\bullet\esm}_{n-1}\ar@{<-}[r]^{\alpha_{n-1}}&{\bsm\bullet\esm}_n \ar[r]^{\alpha_{n-2}}&{\bsm\bullet\esm}_{n-2}\ar[r]^{\alpha_{n-4}}&\cdots {\bsm\bullet\esm}_3\ar [r]^{\alpha_{1}}&{\bsm\bullet\esm}_1
 }\hskip .5cm $$
Then there exists an indecomposable right $\Lambda$-module $M$ with the dimension vector $\mathbf{dim}M=[1, 1, \cdots, 1]^t$ of length $n$. Also for any indecomposable right $\Lambda$-module $N$ that $M\ncong N$, $\mathit{l}(N)< \mathit{l}(M)$. $M$ is local and $\mathit{rad}(M)$ is not local. Then by Theorem \ref{2}, $M$ is $(n-1)$-factor serial. Therefore $\Lambda$ is right $(n-1)$-Nakayama.

$(b).$ In this case, by the hypothesis, $Q$ is of the form
 $$\hskip .5cm \xymatrix{
_1{\bsm\bullet\esm}\cdots \ar[r]&{\bsm\bullet\esm}_b\ar@{<-}[r]&\cdots {\bsm\bullet\esm}_n
 }\hskip .5cm $$
Then there exists an indecomposable right $\Lambda$-module $M$ with the dimension vector $\mathbf{dim}M=[1, 1, \cdots, 1]^t$ of length $n$. Also for any indecomposable right $\Lambda$-module $N$ that $M\ncong N$, $\mathit{l}(N)< \mathit{l}(M)$. Since $Q$ has at least two sources, $M$ is not local and by Proposition \ref{4}, $M$ is $n$-factor serial. Then $\Lambda$ is right $n$-Nakayama.
\end{proof}

Proposition \ref{14} shows that for any $n\in \mathbb{N}$, there exists a right $n$-Nakayama algebra of type $\mathbb{A}_n$.

\begin{proposition}\label{24}
Let $\Lambda=KQ$ be a representation-finite hereditary $K$-algebra. Then the following statements hold:
\begin{itemize}
\item[$(a)$] If the underlying graph $\overline{Q}$ of $Q$ is a Dynkin diagram $\mathbb{D}_n$ for some $n\geq 4$, then $\Lambda$ is a $(2n-3)$-Nakayama algebra.
\item[$(b)$] If the underlying graph $\overline{Q}$ of $Q$ is a Dynkin diagram $\mathbb{E}_6$, then $\Lambda$ is an $11$-Nakayama algebra.
\item[$(c)$] If the underlying graph $\overline{Q}$ of $Q$ is a Dynkin diagram $\mathbb{E}_7$, then $\Lambda$ is a $17$-Nakayama algebra.
\item[$(d)$] If the underlying graph $\overline{Q}$ of $Q$ is a Dynkin diagram $\mathbb{E}_8$, then $\Lambda$ is a $29$-Nakayama algebra.
\end{itemize}
\end{proposition}

\begin{proof}$(a).$ By using the Gabriel's Theorem, there exists an indecomposable right $\Lambda$-module $M$ that $\mathit{l}(M)=2n-3$ and for any indecomposable right $\Lambda$-module $N$ that $M\ncong N$, $\mathit{l}(M)\geq \mathit{l}(N) $. $M$ is not local and so by Proposition \ref{4}, $M$ is a $(2n-3)$-factor serial right $\Lambda$-module. Therefore $\Lambda$ is right $(2n-3)$-Nakayama. The category of finite dimensional left $\Lambda$-modules is equivalent to the category of $mod \Lambda^{op}$, where $\Lambda^{op}\cong KQ^{op}$ and $Q^{op}$ is the opposite of the quiver $Q$. Then the similar argument shows that $\Lambda$ is also left $(2n-3)$-Nakayama and the result follows.

By the similar argument we can prove $(b), (c)$ and $(d)$.
\end{proof}

Let $\Lambda=KQ$ be a representation-finite hereditary $K$-algebra. If $\overline{Q}$ is a Dynkin diagram $\mathbb{A}_n$, Proposition \ref{14} shows that the Nakayama type of $\Lambda$ is depends on the orientation of $Q$. But in the other cases ($\overline{Q}$ is one of the Dynkin diagrams $\mathbb{D}_m, \mathbb{E}_6, \mathbb{E}_7$ and $\mathbb{E}_8$), the Nakayama type of $\Lambda$ dose not depend on the orientation of $Q$.

\section{right $n$-coNakayama algebras}

In this section we define right $n$-coNakayama algebras and by using the standard duality we obtain the connection between right $n$-coNakayama algebras and left $n$-Nakayama algebras.

\begin{definition}
Let $\Lambda$ be an artin algebra and $M$ be a right $\Lambda$-module  of length $l$.
\begin{itemize}
\item[$(a)$] $M$ is called $1$-cofactor serial (uniserial) if $M$ has a unique composition series.
\item[$(b)$] Let $ l\geq n>1 $. We say that $M$ is $n$-cofactor serial if $\mathit{soc}^{l-n}(M)$ is uniserial and $\mathit{soc}^{l-n+1}(M)$ is not uniserial.
\end{itemize}
\end{definition}

\begin{definition}
We say that an artin algebra $\Lambda$ is right $n$-coNakayama if every indecomposable right $\Lambda$-module is $i$-cofactor serial for some $1 \leqslant i \leqslant n$ and there exists at least one indecomposable $n$-cofactor serial right $\Lambda$-module. We can define left $n$-coNakayama algebras analogously.
\end{definition}

Let $R$ be a commutative artinian ring and $\Lambda$ be an artin $R$-algebra. Let $\{S_1,\cdots, S_n\}$ be the complete set of isomorphism classes of simple right $R$-modules, $I(S_i)$ be the injective envelope of $S_i$ and $J=\bigsqcup_{i=1}^{n}I(S_i)$. Then the contravariant $R$-functor $D=Hom_R(-,J):mod \Lambda\rightarrow mod (\Lambda^{op})$ is a duality, which is called standard duality \cite[Theorem II.3.3]{ARS}.

\begin{proposition}\label{12}
Let $\Lambda$ be an artin algebra, $M$ be a right $\Lambda$-module and $\mathit{D}$ be the standard duality. Then
 $M$ is an $n$-factor serial left $\Lambda$-module if and only if $\mathit{D}(M)$ is an $n$-cofactor serial right $\Lambda$-module.
\end{proposition}

\begin{proof}
For each positive integer $i\geq 1$ we have $\frac{\mathit{D}(M)}{\mathit{rad}^i(\mathit{D}(M))}\cong \mathit{D}(\mathit{soc}^{i}(M))$ (see for example the proof of the \cite[Proposition V.1.3.]{As}). Then the result follows by the definition.
\end{proof}

\begin{theorem}\label{13}
Let $\Lambda$ be an artin algebra. Then $\Lambda$ is left $n$-Nakayama if and only if $\Lambda$ is right $n$-coNakayama.
\end{theorem}

\begin{proof}
It follows from Proposition \ref{12}.
\end{proof}

\begin{lemma} Let $\Lambda$ be an artin algebra, $M$ be a non-uniserial right $ \Lambda$-module and $n>1$ be a positive integer. Then the following conditions are equivalent:
\begin{itemize}
\item[$  \left(  \mathrm{a} \right)   $] $M$ is an $n$-cofactor serial right $\Lambda$-module.
\item[$ \left( \mathrm{b}  \right)$] $n-1=\max\{n_i|\text{$M_i$ is an ${n_i}$-cofactor serial maximal submodule of $M$} \}$.
\end{itemize}
\end{lemma}

\begin{proof}$(a)\Longrightarrow (b).$ Assume that $M$ is an $n$-cofactor serial right $\Lambda$-module. By Proposition \ref{12}, $\mathit{D}(M)$ is left $n$-factor serial and by Theorem \ref{6} there exists a simple submodule $S$ of $\mathit{D}(M)$ such that $\frac{\mathit{D}(M)}{S} $ is $(n-1)$-factor serial.
Then we have the following exact sequence of right $\Lambda$-modules
\begin{center}
$0\rightarrow \mathit{D}(\frac{\mathit{D}(M) }{S}) \rightarrow M  \rightarrow \mathit{D}(S) \rightarrow 0$
\end{center}
It is clear that $\mathit{D}(S)$ is a simple right $\Lambda$-module and so $\mathit{D}(\frac{\mathit{D}(M) }{S})$ is a maximal submodule of $M$. By Proposition \ref{12}, $\mathit{D}(\frac{\mathit{D}(M) }{S})$ is $(n-1)$-cofactor serial. Now let $N$ be a proper submodule of $M$, then we have a proper epimorphism  $\mathit{D}(M) \rightarrow \mathit{D}(N)$ of left $\Lambda$-modules. By Lemma \ref{5}, $ \mathit{D}(N) $ is $s$-factor serial for some $s\leq n-1$. Therefore $N$ is an $s$-cofactor serial right $\Lambda$-module and the result follows.

$(b)\Longrightarrow (a).$ Assume that $n-1=\\\max\{n_i|\text{$M_i$ is an ${n_i}$-cofactor serial maximal submodule of $M$} \}$. Then there is a maximal submodule $M_i$ of $M$ that $M_i$ is an $(n-1)$-cofactor serial right $\Lambda$-module. Thus there is a proper epimorphism $\mathit{D}(M)\longrightarrow \mathit{D}(M_i)$ of left $\Lambda$-modules. By Lemma \ref{5}, $\mathit{D}(M)$ is a $t$-factor serial left $\Lambda$-module for some $t\geq n$. If $t>n$, then by Proposition \ref{12}, $M$ is a $t$-cofactor serial right $\Lambda$-module. Thus by the first part, $t-1=\\\max\{n_i|\text{$M_i$ is an ${n_i}$-cofactor serial maximal submodule of $M$} \}$, which is a contradiction. Then $t=n$ and the result follows.
\end{proof}

\section{right 2-Nakayama algebras}

In this section we describe the representation theory of the right $2$-Nakayama algebras. We give a complete list of their non-isomorphic indecomposable modules. As a consequence we compute all almost split sequences.

\begin{lemma}\label{17}
Let $\Lambda$ be an artin algebra and $M$ be a $2$-factor serial right $\Lambda$-module. Then $\mathit{soc}(M)=S_1\oplus S_{2}$ that $S_1$ and $S_2$ are simple submodules of $M$.
\end{lemma}

\begin{proof}
By Remark \ref{3}, $\mathit{l}(\mathit{soc}(M))\leq 2$. We show that $\mathit{soc}(M)$ is not simple.
Assume that $\mathit{soc}(M)=S$ is a simple right $\Lambda$-module. Then by Theorem \ref{6}, $\frac{M}{S}$ is uniserial. Thus $\frac{M}{S}$ has a unique composition series of the form
$$0= \frac{S}{S}\subset \frac{M_{l-1}}{S}\subset ....\subset \frac{M_1}{S}\subset \frac{M}{S}$$
Then $M$ has a unique composition series of the form
$$0\subset S\subset M_{l-1}\subset...\subset M_1\subset M $$ which gives a contradiction. Then $\mathit{l}(\mathit{soc}(M))=2$ and the result follows.
\end{proof}

\begin{lemma}\label{18}
Let $\Lambda$ be an artin algebra and $M$ be a right $\Lambda$-module of Loewy length $t\geq 2$.  Then the following conditions are equivalent:
\begin{itemize}
\item[$(a)$] $M$ is a $2$-factor serial right $\Lambda$-module.
\item[$(b)$] For every $0\leqslant i\leqslant t-2$, $\textit{rad}^{i} \left(M \right) $ is local and
$\textit{rad}^{t-1}\left(M \right)=\textit{soc}\left(M \right)=S_1\oplus S_2  $, where $S_1$ and $ S_2$ are simple submodules of $M$.
\item[$(c)$] $\textit{l}\left(M \right)=t+1 $ and $\textit{soc}\left(M \right)=\textit{rad}^{t-1}\left(M \right)=S_1\oplus S_2 $, where $S_1$ and $S_2$ are simple submodules of $M$.
\end{itemize}
\end{lemma}

\begin{proof} $(a)\Longrightarrow (b).$ First we show that for every $0\leqslant i\leqslant t-2$,
$\textit{rad}^i \left(M\right)  $  is local. Assume that there exists $0\leqslant i\leqslant t-2$, such that $\textit{rad}^i \left(M\right)  $ is not local. Let $M_1$ and $M_2$ be two maximal submodules of $rad^i(M)$. Let $N:=M_1\bigcap M_2 $. Then $N$ is a maximal submodule of $M_i$ for $i=1, 2$.
This implies that there exist the following composition series for $M$
$$0\subset \cdots\subset N \subset M_1\subset \textit{rad}^i \left(M\right) \subset\cdots \subset rad(M)\subset M$$
$$0\subset \cdots\subset N \subset M_2\subset \textit{rad}^i \left(M\right) \subset\cdots\subset rad(M)\subset M$$
Since $\frac{M}{N}$ has two different composition series, $\frac{M}{N}$ is not uniserial. On the other hand by Remark \ref{R6}, $\frac{M}{N}$ is uniserial which is a contradiction. Then for each $0\leq i\leq t-2$, $rad^i(M)$ is local. By Lemma \ref{17}, $ \textit{soc}(M)=S_1\oplus S_2$. Now we show that $\textit{rad}^{t-1}\left(M \right)=\textit{soc}(M)$. It is clear that
$\textit{rad}^{t-1}\left(M \right)\subseteq \textit{soc}(M)$. If $\textit{rad}^{t-1}\left(M \right)\neq \textit{soc}(M)$, then $\textit{soc}(M)=\textit{rad}^{i}\left(M \right)$ for some $0\leq i\leq t-2$, which is a contradiction. Then $\mathit{rad}^{t-1}(M)=\mathit{soc}(M)$ and the result follows.

$(b)\Longrightarrow (c).$ It is obvious.

$(c) \Longrightarrow (a).$ Since $\frac{M}{S_1}$ and $\frac{M}{S_2}$ are uniserial, the result follows by Theorem \ref{6}.
\end{proof}

\begin{lemma}
Let $\Lambda$ be an artin algebra and $M$ be a $2$-factor serial right $\Lambda$-module of length $l$. Then $M$ is either indecomposable or $M=S_1\oplus S_2$ where $S_1$ and $S_2$ are simple right $\Lambda$-modules.
\end{lemma}

\begin{proof}
Assume that $M$ is not indecomposable. By Corollary \ref{15}, $M$ is $l$-factor serial. Then $l=2$ and the result follows.
\end{proof}

An artin algebra $\Lambda$ is called left serial if every indecomposable injective right $\Lambda$-module is uniserial.

\begin{proposition}\label{19}
Let $\Lambda$ be an artin algebra. If $\Lambda $ is a right $2$-Nakayama algebra, then $\Lambda $ is left serial.
\end{proposition}

\begin{proof}
Let $I$ be an indecomposable injective right $\Lambda$-module. Since $\mathit{soc}(I)$ is simple, by Lemma \ref{17}, $I$ is uniserial and the result follows.
\end{proof}

\begin{proposition}\label{PP}
Let $\Lambda$ be a right $2$-Nakayama artin algebra and $M$ be an indecomposable $2$-factor serial right $\Lambda$-module. Then $M$ is projective, $\mathit{l}(M)=3$, $\mathit{ll}(M)=2$ and $\mathit{rad}(M)=\mathit{soc}(M)$.
\end{proposition}

\begin{proof} By Lemma \ref{18}, $M$ is local. Then there exists an indecomposable projective right $\Lambda$-module $P$ that $\frac{P}{N}\cong M$, for some submodule $N$ of $P$. If $N\neq 0$, then by Lemma \ref{5}, $P$ is $t$-factor serial for some $t\geq 3$ which gives a contradiction. Then $M$ is a projective right $\Lambda$-module and $\Lambda$ is of local type. Thus by Proposition $1.4$ of \cite{Asa}, $\mathit{rad}(M)$ is decomposable and so $\mathit{rad}(M)=\mathit{soc}(M)=S_1\oplus S_2$, where $S_i$ is a simple submodule of $M$ for each $i=1, 2$. Therefore $\mathit{l}(M)=3$ and $\mathit{ll}(M)=2$.
\end{proof}

\begin{proposition}\label{P20}
Let $\Lambda$ be a right $2$-Nakayama artin algebra and $M$ be an indecomposable right $\Lambda$-module. Then there exists an indecomposable projective right $\Lambda$-module $P$ and a submodule $N$ of $P$ such that $M\cong \frac{P}{N}$. More precisely, if $P$ is an indecomposable projective uniserial module, then $M\cong \frac{P}{\textit{rad}^i (P)}$ for some $1\leq i\leq\textit{ll}(P)$ and if $P$ is an indecomposable projective $2$-factor serial module that $\mathit{rad}(P)=\textit{soc}(P)=S_1\bigoplus S_2$, then $M$ is either isomorphic to $P$ or isomorphic to $\frac{P}{\textit{rad}(P)}$ or isomorphic to $\frac{P}{S_i}$ for some $1\leq i\leq 2$.
\end{proposition}

\begin{proof} Let $M$ be an indecomposable non-projective right $\Lambda$-module. By Proposition \ref{PP}, $M$ is uniserial. Then there exists an indecomposable projective right $\Lambda$-module $P$ that $M\cong \frac{P}{N}$ for some submodule $N$ of $P$. If $P$ is uniserial, then $M\cong \frac{P}{\mathit{rad}^{i}(P)}$ for some $1\leq i < \mathit{ll}(P)$. Now assume that $P$ is an indecomposable projective $2$-factor serial right $\Lambda$-module. Then by using Lemma \ref{18} and Proposition \ref{PP}, $M$ is either isomorphic to $\frac{P}{\textit{rad} (P)}$ or isomorphic to $\frac{P}{S_i}$ for some $1\leq i\leq 2$.
\end{proof}

\begin{remark} \begin{itemize} \item[$(1)$] Let $\Lambda$ be a right $2$-Nakayama artin algebra and $M$ be a $2$-factor serial indecomposable right $\Lambda$-module. By Proposition \ref{PP},  $ll(M)=2$, $l(M)=3$ and $\mathit{soc}(M)=S_1\oplus S_2$ where $S_1$ and $S_2$ are simple right $\Lambda$-modules. $\mathit{D}(M)$ has maximal submodules which are isomorphic to $\mathit{D}(\frac{M}{S_1})$ and $\mathit{D}(\frac{M}{S_2}) $. By Proposition \ref{4}, $\mathit{D}(M)$ is $3$-factor serial left $\Lambda$-module and $\Lambda$ is a left $3$-Nakayama algebra. Then there does not exist a $2$-Nakayama artin algebra.
\item[$(2)$] Let $\Lambda$ be a right $2$-Nakayama artin algebra. By Proposition \ref{PP}, there exists an indecomposable projective $2$-factor serial right $\Lambda$-module $P$. By Lemma $\ref{17}$, $\mathit{soc}(P)$ is not simple and so $P$ is not injective. Then there does not exist a right $2$-Nakayama self-injective artin algebra.
\end{itemize}
\end{remark}

\begin{theorem}
Let $\Lambda$ be an artin algebra which is not Nakayama. Then $\Lambda$ is right $2$-Nakayama if and only if every indecomposable non-projective right $\Lambda$-module is uniserial.
\end{theorem}

\begin{proof} Assume that $\Lambda$ is right $2$-Nakayama. Then by Proposition \ref{PP}, every indecomposable non-projective right $\Lambda$-module is uniserial.
Now suppose that every indecomposable non-projective right $\Lambda$-module is uniserial. Let $P$ be an indecomposable projective right $\Lambda$-module which is not uniserial. We show that $P$ is a $2$-factor serial right $\Lambda$-module. Assume that $P$ is a $t$-factor serial right $\Lambda$-module. Then by Corollary \ref{7}, there is a sequence of proper epimorphisms and right $\Lambda$-modules such
 \begin{center}
$ M_0\buildrel {f_0} \over \longrightarrow M_1 \buildrel {f_{1}} \over \longrightarrow \cdots \buildrel {f_{t-3}} \over \longrightarrow M_{t-2}\buildrel {f_{t-2}} \over \longrightarrow M_{t-1}$
 \end{center}
that $M_0=P$ and $M_i$ is a $(t-i)$-factor serial right $\Lambda$-module for each $1\leq i \leq t-1$. Since $f_0$ is a proper epimorphism and $top(M_0)$ is simple, $M_1$ is an indecomposable right $\Lambda$-module. If $M_1$ is not uniserial, then $M_1$ is indecomposable projective. Thus $M_1$ is a direct summand of $M_0$ which is a contradiction. Therefore $M_1$ is uniserial and so $P=M_0$ is $2$-factor serial. Then $\Lambda$ is right $2$-Nakayama and the result follows.
\end{proof}

\begin{corollary} Let $\Lambda$ be an artin algebra.
\begin{itemize}
\item[$(i)$] If every indecomposable non-projective right $\Lambda$-module is uniserial, then $\Lambda$ is either Nakayama or right $2$-Nakayama.
\item[$(ii)$] If every indecomposable non-injective right $\Lambda$-module is uniserial, then $\Lambda$ is either Nakayama or left $2$-Nakayama.
\end{itemize}
\end{corollary}

Now we show that how to compute all almost split sequences in the module category of a right $2$-Nakayama artin algebra $\Lambda$.

\begin{theorem} Let $\Lambda$ be a right $2$-Nakayama artin algebra and $M$ be an indecomposable non-projective right $\Lambda$-module.
\begin{itemize}
\item[$(a)$] If $M\cong \frac{P}{\textit{rad}^t(P)}  $ for some $ 1\leq t < \textit{ll}(P) $, where $P$ is an uniserial projective right $\Lambda$-module, then the sequence
\begin{equation*}
0\longrightarrow \frac{\textit{rad}(P)}{\textit{rad}^{t+1}(P)}\buildrel{
\begin{bmatrix}
q\\
i\\
\end{bmatrix}
}\over\longrightarrow \frac{\textit{rad}(P)}{\textit{rad}^t(P)}\oplus \frac{P}{\textit{rad}^{t+1}(P)}\buildrel{[-j,p]}\over \longrightarrow \frac{P}{\textit{rad}^t(P)}\longrightarrow0
\end{equation*}
is an almost split sequence (where $p$ and $q$ are the canonical epimorphisms and $i$ and $j$ are the inclusion homomorphisms).

\item[$(b)$] If  $M\cong \frac{P}{S_t}$ for some $1\leq t\leq 2$, where $P$ is an indecomposable $2$-factor serial projective right $\Lambda$-module that $\mathit{rad}(P)=\textit{soc}(P)=S_1\bigoplus S_{2}$, then the sequence
  \begin{equation*}
      0\rightarrow S_t\buildrel{i_1}\over\longrightarrow P\buildrel{p_1}\over\longrightarrow \frac{P}{S_t}\rightarrow 0
     \end{equation*}
     is an almost split sequence (where $i_1$ is an inclusion homomorphism and $p_1$ is a canonical epimorphism).
 \item[$(c)$] If $M\cong \dfrac{P}{\textit{rad}(P)}$, where $P$ is an indecomposable $2$-factor serial projective right $\Lambda$-module that $\mathit{rad}(M)=\textit{soc}(P)=S_1\bigoplus S_{2}$, then the sequence
  \begin{equation*}
  0\rightarrow P\buildrel{
  \begin{bmatrix}
  \pi_1 \\
  \pi_2\\
  \end{bmatrix}
   }\over\longrightarrow \frac{P}{S_1}\oplus \frac{P}{S_2}\buildrel{[\pi_3,\pi_4]} \over \longrightarrow \frac{P}{\textit{rad}(P)}\rightarrow 0
    \end{equation*}
 is an almost split sequence (where $\pi_1$, $\pi_2$, $\pi_3$ and $\pi_4$ are canonical epimorphisms).
 \end{itemize}
 \end{theorem}

 \begin{proof} Put $g_1=[-j,p]$ and $g_2=[\pi_3,\pi_4]$. It is easy to see that the given sequences are exact, not split and have indecomposable end terms. Thus it suffices to prove that homomorphisms $g_1$, $p_1$ and $ g_2 $ are right almost split.

$(a).$ Let $V$ be an indecomposable right $\Lambda$-module and $v:V\longrightarrow \frac{P}{\textit{rad}^{t}(P)}$ be a non-isomorphism. If $ v$ is not surjective, then $\mathrm{Im}(v)$ is contained in the unique maximal submodule $\textit{rad}(M)=\frac{\textit{rad}(P)}{\textit{rad}^{t}(P)}$ of $M$. Thus the homomorphism\\
 $$\begin{bmatrix}
-v\\
0\\
\end{bmatrix}
:V\longrightarrow \frac{\textit{rad}(P)}{\textit{rad}^t(P)}\oplus \frac{P}{\textit{rad}^{t+1}(P)} $$
satisfies $g_{1}\begin{bmatrix}
-v\\
0\\
\end{bmatrix}
=v$. \\
If $v$ is surjective, then $V\cong \frac{P}{\mathrm{rad}^{s}(P)}$ for some $s>t$ and hence there exists a homomorphism
$v^{'}:V\longrightarrow \frac{\textit{rad}(P)}{\textit{rad}^t(P)}\oplus \frac{P}{\textit{rad}^{t+1}(P)}$
that $g_{1}v^{'}=v$.

$(b).$ Let $V$ be an indecomposable right $\Lambda$-module and $v:V\longrightarrow \frac{P}{S_{t}}$ be a non-isomorphism.
If $v$ is surjective, then $V\cong P$. Hence there exists an isomorphism $\alpha:V\longrightarrow P$ such that $p_1 \alpha=v$. If $v$ is not surjective, then $\mathrm{Im}(v)\cong  \frac{\mathit{rad}(P)}{S_{t}}$ which is a simple submodule of $P$. Therefore there exists a homomorphism $\alpha: V\longrightarrow P$ such that $p_1\alpha=v$.

$(c).$ Let $V$ be an indecomposable right $\Lambda$-module and $v:V\longrightarrow \frac{P}{\textit{rad}(P)}$ be a non-isomorphism. Since $ \frac{P}{\textit{rad}(P)} $ is a simple right $\Lambda$-module, $v$ is an epimorphism and $V$ is either isomorphic to $P$ or isomorphic to $\frac{P}{S_t}$ for some $t=1,2$. Then there exists a homomorphism
 $$v^{'}:V\longrightarrow \frac{ P}{S_1}\oplus \frac{P}{S_2} $$
such that $g_2v^{'}=v$.
\end{proof}

\begin{proposition}\label{28}
Let $\Lambda=\frac{KQ}{I}$ be a basic and connected finite dimensional $K$-algebra. If $\Lambda$ is right $2$-Nakayama, then for any vertex $a\in Q_0$, $|a^+|\leq 2$.
\end{proposition}

\begin{proof} By Proposition \ref{19}, $\Lambda$ is left serial and so by \cite[Theorem $V.2.6$]{As}, for any vertex  $a\in Q_0$, $|a^-|\leq1$. Assume that there exists a vertex $a_0\in Q_0$, such that $|a_0^+|\geq 3$. We have two cases:
\begin{itemize}
\item Case $(1)$: There is no cycle in $Q$ which is contains the vertex $a_0$\\
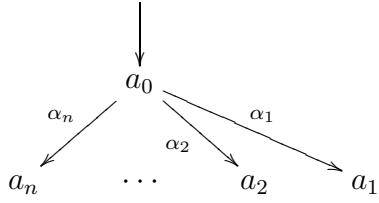
$$\hskip .5cm \xymatrix{
&\ar[d]\\
&{a_0}\ar[dl]_{\alpha_n}\ar[dr]_{\alpha_2}\ar[drr]^{\alpha_1}\\
{a_n}&\cdots&{a_2}&{a_1}
}\hskip .5cm$$

\item Case $(2)$: There exists at least one cycle in $Q$ which is contains the vertex $a_0$.

$$\hskip .5cm 
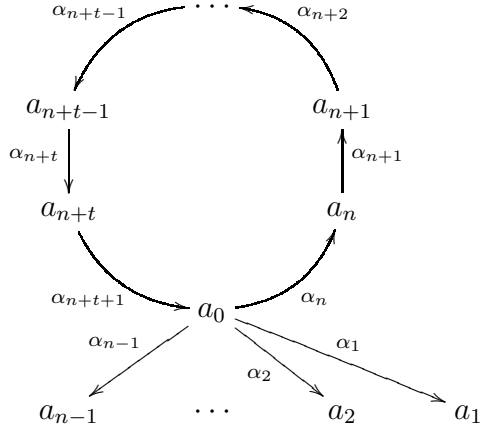
\xymatrix{
&\cdots\ar @/_1pc/[dl]_{\alpha_{n+t-1}}\\
{a_{n+t-1}}\ar [d]_{\alpha_{n+t}}&&{a_{n+1}}\ar@/_1pc/[ul]_{\alpha_{n+2}}\\
{a_{n+t}}\ar @/_1pc/[dr]_{\alpha_{n+t+1}}&&{a_{n}}\ar [u]_{\alpha_{n+1}}\\
&{a_0}\ar[dl]_{\alpha_{n-1}}\ar[dr]_{\alpha_2}\ar[drr]^{\alpha_1}\ar@/_1pc/[ur]_{\alpha_{n}}\\
{a_{n-1}}&\cdots&{a_2}&{a_1}
}\hskip .5cm$$
In the case $\left( 1\right) $, the indecomposable projective representation $ P(a_0)=(P_j, \phi_{\alpha})_{(j\in Q_0,\ \alpha \in Q_1)}$ has at least three subrepresentations of dimension $d \geq 2$ of the forms $M_i=(M_{ij}, \psi_{i\alpha})_{(j\in Q_0,\ \alpha\in Q_1)}$ for $1\leq i\leq n$ such that
\begin{center}
$M_{ij}=\begin{cases}
0& j=a_i,\,  a_0,\\
P_j& otherwise,
\end{cases}$
\end{center}
 and
 \begin{center}
 $\psi_{i\alpha}=\begin{cases}
 0& \alpha=\alpha_i,\,\ 1\leq i\leq n,\\
  \phi_{\alpha} & otherwise.
 \end{cases}$
 \end{center}
Then by Lemma \ref{18}, $P(a_0)$ is not right $2$-factor serial. Also $P(a_0)$ is not uniserial which is a contradiction.\\
In the case $\left( 2\right) $, the indecomposable projective representation $P(a_0)=(P_j, \phi_{\alpha})_{(j\in Q_0,\ \alpha \in Q_1)}$ has at least two subrepresentations of dimension $d\geq 2$ of the forms $M_i=(M_{ij}, \psi_{i\alpha})_{(j\in Q_0,\ \alpha \in Q_1)}$ for $1\leq i\leq n-1$ such that
\begin{center} $M_{ij}=\begin{cases}
K^{\lambda_j-1}&  j=a_i,\,  a_0,\\
P_j& otherwise,
\end{cases}$
\end{center}
 and
 \begin{center}
 $\psi_{i\alpha}=\begin{cases}
\phi_\alpha\mid _{K^{\lambda_j-1}}  & \alpha=\alpha_i, \\
  \phi_{\alpha} & otherwise,
 \end{cases}$
 \end{center}
where $P_j=K^{\lambda_j}$ for some positive integer $\lambda_j$. Then by Lemma \ref{18}, $P(a_0)$ is not right $2$-factor serial. Also $P(a_0)$ is not uniserial which is a contradiction.
\end{itemize}
\end{proof}

\begin{remark}\label{29}
Let $Q$ be a finite quiver, $I$ be an admissible ideal of $Q$, $Q^{'}$ be a subquiver of $Q$ and $I^{'}$ be an admissible ideal of $Q^{'}$ which is restriction of $I$ to $Q^{'}$. Then there exists a fully faithful embedding functor $F:rep_K(Q^{'}, I^{'})\longrightarrow rep_K(Q, I)$.

\end{remark}

\begin{theorem}\label{30} Let $ \Lambda=\frac{KQ}{I} $ be a basic and connected finite dimensional $K$-algebra.
Then $\Lambda $ is right $2$-Nakayama if and only if $\left(Q, I\right)$ satisfies the following conditions:
\begin{itemize}
\item[$\left( \mathrm{i}\right) $] For every $a\in Q_0$, $\mid a^+\mid \leqslant  2$ and $\mid a^-\mid \leqslant 1$.
\item[$\left( \mathrm{ii}\right) $] For every $\alpha \in Q_1$, there is at most one arrow $\beta$ that $\alpha\beta\not\in I$.
\item [$\left( \mathrm{iii}\right) $] There exists $a\in Q_0$ that $\mid a^+\mid =2$.
\item [$\left(\mathrm{iv}\right) $] For every non-lazy paths $\sigma$ and $ \gamma$, that $\sigma$ and $ \gamma$ have the same sources and neither $\sigma$ is a subpath of $\gamma$ nor $\gamma$ is a subpath of $\sigma$,
  $length\left( \sigma \right) +length\left( \gamma\right) =2$.
\end{itemize}
 \end{theorem}

 \begin{proof} Assume that $\Lambda$ is a right $2$-Nakayama algebra. By Proposition \ref{19}, $\Lambda$ is a left serial algebra. Then by \cite[Theorem V.2.6.]{As}, for any $a\in Q_0$, $|a^-|\leq1$.
 By Proposition \ref{28}, for every $a\in Q_0$,  $\mid a^{+}\mid \leq 2$. It is clear that there exists $a\in Q_0$ that $\mid a^+\mid =2$.
 Assume that the condition $\left(\mathrm{ii}\right) $ does not hold. Then we have two cases.
\begin{itemize}
 \item Case $(1)$: The Algebra $\Lambda'=KQ'$ given by the quiver $Q'$
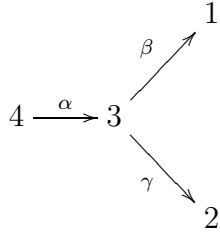
 \begin{center}
 $$\hskip .5cm \xymatrix{
&&{1}\ar @{<-}[dl]_{\beta}\\
{4}\ar [r]^{\alpha}&{3}\ar [dr]_{\gamma}\\
&&{2}}\hskip .5cm$$
\end{center}
is a subalgebra of $\Lambda$.
 Then there is an indecomposable representation $M$ of $Q'$ such that $\mathbf{dim} M=[1, 1, 2, 1]^t$. $M$ is not local and by Proposition \ref{4}, $M$ is a $5$-factor serial right $\Lambda'$-module. Therefore by using Remark \ref{29}, there exists a $5$-factor serial right $\Lambda$-module which is a contradiction.
  \item Case $(2)$: The Algebra $\Lambda''=\frac{KQ''}{I''}$ given by the quiver $Q''$
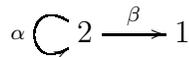
  \begin{center}
$\begin{matrix}\xymatrix{{2}\ar@(ul,dl)_{\alpha} \ar[r]^{\beta}&{1}}
\end{matrix}\hskip.5cm$
\end{center}
bounded by ${\alpha}^3=0$ and $\alpha^{2}\beta=0$ is a subalgebra of $\Lambda$. There is an indecomposable representation $M$ of $(Q'', I'')$ such that $\mathbf{dim} M=[2, 4]^t$ and $M$ is not local. Then by Proposition \ref{4}, $M$ is a $6$-factor serial right $\Lambda''$-module. Therefore by using Remark \ref{29}, there exists a $6$-factor serial right $\Lambda$-module which is a contradiction.
\end{itemize}
Now suppose that the condition $\left( \mathrm{iv}\right) $ does not hold. Then we have four cases.
\begin{itemize}
 \item Case $(1)$: The Algebra $\Lambda_1=KQ_1$ given by the quiver $Q_1$
 \begin{center}
\begin{tikzpicture}[main_node/.style={
                                         }
                                         ]
\tikzset{edge/.style = {->,> = latex'}}

 \node[main_node] (4) at (0, 0) {4};
    \node[main_node] (3) at (1,1) {3};
      \node[main_node] (2) at (1,0) {2};
    \node[main_node] (1) at (2, 2)  {1};

    \draw[->]  (4) edge node[above] {\(\alpha\)} (3);
    \draw[->]  (4) edge node[below] {\(\beta\)} (2);
     \draw[->]  (3) edge node[above] {\(\gamma\)} (1);

\end{tikzpicture}
\end{center}
is a subalgebra of $\Lambda$. Then there is an indecomposable representation $M$ of $Q_1$ such that $\mathbf{dim} M=[1, 1, 1, 1]^t$ and $M$ is a $3$-factor serial right $\Lambda_1$-module which gives a contradiction.
  \item Case $(2)$: The Algebra $\Lambda_2=\frac{KQ_2}{I_2}$ given by the quiver $Q_2$
  \begin{center}
$\begin{matrix}\xymatrix{{2}\ar@(ul,dl)_{\alpha} \ar[r]^{\beta}&{1}}
\end{matrix}\hskip.5cm$
\end{center}
bounded by ${\alpha}^3=0$ and $\alpha\beta=0$ is a subalgebra of $\Lambda$. Let $M$ be an indecomposable representation of $(Q_2, I_2)$ of dimension vector $\mathbf{dim} M=[1, 3]^t$. Then $M$ is a $3$-factor serial right $\Lambda_2$-module which gives a contradiction.
\item Case $(3)$: The Algebra $\Lambda_3=\frac{KQ_3}{I_3}$ given by the quiver $Q_3$
\begin{center}
$\begin{matrix}\xymatrix{{2}\ar@(ul,dl)_{\alpha} \ar[r]^{\beta}&{1}}
\end{matrix}\hskip.5cm$
\end{center}
bounded by ${\alpha}^2=0$ is a subalgebra of $\Lambda$. Let $M$ be an indecomposable representation of $(Q_3, I_3)$ of dimension vector $\mathbf{dim} M=[2, 2]^t$. Then $M$ is a $3$-factor serial right $\Lambda_3$-module which gives a contradiction.
  \item Case $(4)$: The Algebra $\Lambda_4=\frac{KQ_4}{I_4}$ given by the quiver $Q_4$
  \begin{center}
$\begin{matrix}\xymatrix{{3}\ar@(ul,dl)_{\alpha} \ar[r]^{\beta}&{2}\ar[r]^{\gamma}&{1}}
\end{matrix}\hskip.5cm$
\end{center}
bounded by ${\alpha}^2=0$ and $\alpha\beta=0$ is a subalgebra of $\Lambda$. Let $M$ be an indecomposable representation of $(Q_4, I_4)$ of dimension vector $\mathbf{dim} M=[1, 1, 2]^t$. Then $M$ is a $3$-factor serial right $\Lambda_4$-module which gives a contradiction.
\end{itemize}

Conversely, assume that $\Lambda=KQ/I$ such that $(Q, I)$ satisfies the conditions $(i), (ii), (iii)$ and $(iv)$. It is clear that $\Lambda$ is a string algebra. By the main Theorem of \cite{BR}, every indecomposable right $\Lambda$-module is either a string module $M(w)$ for some $w\in \widetilde{\mathcal{S}}(\Lambda)$ or a band module $M(v, m, \varphi)$ for some $v\in \widetilde{\mathcal{B}}(\Lambda)$, $m\geq 1$ and $\varphi\in Aut(K^m)$. By the conditions $\left( \mathrm{i}\right)$, $\left( \mathrm{ii}\right)$ and $\left( \mathrm{iv}\right)$, $\mathcal{B}(\Lambda)=\varnothing$ and for any $w\in \widetilde{\mathcal{S}}(\Lambda)$, $w=w_1^{+1}...w_n^{+1}$ for some positive integer $n$ or $w=w_1^{-1}w_2^{+1}$. If $w=w_1^{+1}...w_n^{+1}$, then $M(w)$ is uniserial and if $w=w_1^{-1}w_2^{+1}$, then $M(w)$ is $2$-factor serial. By the condition $\left( \mathrm{iii}\right)$, there exists at least one string module $M(w)$, that $w=w_1^{-1}w_2^{+1}$. Therefore $\Lambda$ is right $2$-Nakayama and the result follows.
\end{proof}


\section*{acknowledgements}
The authors would like to thank to the referee for a careful reading of this paper. Also we are thankful to Professor C. M. Ringel for many helpful
discussions, comments and suggestions. In particular, he suggested to use the word $n$-factor serial module. The research of the first
author was in part supported by a grant from IPM (No. 96170417).

\end{document}